\documentclass{amsart}
\usepackage{amssymb}
\usepackage{appendix}
\usepackage[symbol]{footmisc}

\usepackage{graphicx}
 \usepackage[all]{xy}
 
\usepackage{xcolor}
\usepackage{boxhandler}
\usepackage{caption}
\usepackage{lipsum}
\newtheorem*{theorem*}{Main Theorem}
\newtheorem{theorem}{Theorem}
\newtheorem*{Lemma}{Reparametrization Lemma}
\newtheorem{lemma}{Lemma}
\newtheorem{prop}[theorem]{Proposition}
\newtheorem{coro}[theorem]{Corollary}
\newtheorem{rem}[theorem]{Remark}
\theoremstyle{definition}
\newtheorem{definition}{Definition}
\newtheorem*{ques}{Question}

\newcommand{\Ima}{\mathop{\mathrm{Im}}}
\newcommand{\Jac}{\mathop{\mathrm{Jac}}}

\newcommand{\Leb}{\mathop{\mathrm{Leb}}}
\theoremstyle{remark}

\numberwithin{equation}{section}

\begin{document}

\title{Entropy of physical measures for  $C^\infty$ smooth systems}

\author{David Burguet}
\address{Sorbonne Universite, LPSM, 75005 Paris, France}
  \email{david.burguet@upmc.fr}  

\date{June 2018}


\begin{large}
\begin{abstract}
For a $C^\infty$ map on a compact manifold we prove that for a Lebesgue randomly   picked point $x$ there is an empirical measure from  $x$ with entropy larger than or equal to the sum of positive Lyapunov exponents at $x$. This contrasts with the well-known Ruelle inequality.  As a consequence we give some refinement of Tsujii's work \cite{Tsu} relating physical and Sinai-Ruelle-Bowen measures. 
\end{abstract}

\maketitle

\section*{Introduction}
Entropy is a master invariant in dynamical systems, which estimates the dynamical complexity by counting the separated orbits. 
For a differentiable system other dynamical quantities of high interest are  the Lyapunov exponents. They are given by the exponential growth rate of the derivative. Heuristically the first  derivative controls the separation of points (as in the mean value inequality) so that the entropy is always less than or equal to the (sum of positive) Lyapunov exponents. This inequality, due to Ruelle \cite{Rue}, holds at any invariant measure. Moreover the case of equality characterizes the so-called Sinai-Ruelle-Bowen measures for $C^{1+\alpha}$ systems. 

 Here we use a slightly different framework. We do not consider entropy and Lyapunov exponent defined on  invariant measures but on points. For the entropy we let $h(x)$ be the supremum  entropy of the  empirical measures   at a given point $x$.  We may also define a pointwise sum of positive Lyapunov exponents, denoted by $\Sigma\chi^+(x)$, by considering the limsup in the exponential growth  of the derivative at $x$ (see Section 1 for the precise definitions). We then aim to compare $h(x)$ and  $\Sigma\chi^+(x)$ "physically", i.e. for Lebesgue almost every point $x$.  For a $C^\infty$ system we prove quite surprisingly the entropy is physically bounded from below by the sum of positive Lyapunov exponents, i.e.  \[h\geq \Sigma\chi^+\text{ almost surely.}\]

In \cite{Yom} Yomdin introduced tools of semi-algebraic geometry in order to control the local volume growth of $C^\infty$ smooth systems.  In particular it allows him to show that Shub's entropy conjecture holds true in this setting. 
Using a similar approach we manage to control locally not only the volume growth but also the distortion (see also \cite{Bur} and \cite{Burr}). The resulting Reparametrization Lemma of dynamical balls is the key argument in the proof of our Main Theorem. \\

The paper is organized as follows. In the first section we recall the notions of physical, physical-like and Sinai-Ruelle-Bowen measures. We also introduce the  \textit{strong Lyapunov exponents} which provide a new way to estimate the exponential growth of the derivative at a point. Our Main Theorem and its Corollaries are stated and discussed in Section 2. The last two sections are devoted to the proof. Finally we present a counter-example in finite smoothness in the appendix.



\section{Background}

\subsection{Physical measures}
Let $(M,f)$ be a topological system, i.e. $f:M\rightarrow M$ is a continuous map on a compact metrizable space $M$. Fix a metric $\mathsf d$ on $M$. We let $\mathcal{M}(M)$ (resp. $\mathcal{M}(M,f)$) be the set of   Borel probability measures (resp. $f$-invariant). Endowed with the  weak-$*$ topology these sets are compact metrizable spaces. When $(\phi_n:M\rightarrow \mathbb{R})_{n\in \mathbb{N}}$ is a dense countable family of the set of real continuous functions  on $X$ for  the usual supremum norm then the following convex
 metric $\mathfrak{d}$ on $\mathcal{M}(M)$ is compatible with the weak-$*$ topology:
\[\forall \mu,\nu\in \mathcal{M}(M), \ \mathfrak{d}(\mu, \nu):=\sum_{n}\frac{\left| \int \phi_n \,d\nu-\int \phi_n \,d\mu\right|}{2^n(1+\sup_x|\phi_n(x)|)}.\] We will also consider the set  $\mathcal{KM}(M)$ of all nonempty closed subsets of $\mathcal{M}(M)$ with the associated Hausdorff metric $\mathfrak{d}^H$.

  The basin $\mathcal{B}_\mu$ of an invariant measure $\mu\in \mathcal{M}(M,f)$ is the set of point $x\in M$ whose empirical measures $\mu_n^x:=\frac{1}{n}\sum_{0\leq k<n}\delta_{f^kx}$ is converging to $\mu$, when $n$ goes to infinity. According to Birkhof ergodic theorem the set $\mathcal{B}_\mu$ has full $\mu$-measure for an ergodic measure $\mu$. In the following we will always consider a $C^\infty$ smooth compact Riemannian manifold $(M,\|\cdot\|)$ and its induced Riemannian distance $\mathsf d$. The (normalized) volume form inherited from the Riemannian structure will be called the Lebesgue measure and is denoted  by $\Leb$. An invariant  measure is said \textbf{physical} when its basin has positive Lebesgue measure. From the works of Sinai, Ruelle and  Bowen  \cite{Sin,Rue,Bow} it is known that any $C^2$ Axiom A diffeomorphism admits finitely many ergodic physical measures such that the  union of their basin has full Lebesgue measure.

We recall now the concept of physical-like measures \cite{Cat,CH}. For  $x\in M$ we let $pw(x)\subset \mathcal{M}(M,f)$ be the accumulation points of the empirical measures $(\mu_n^x)_n$ at $x$. An invariant measure $\mu\in \mathcal{M}(M,f)$ is said \textbf{physical-like} when for any $\epsilon>0$ the set $\{x, \ \mathfrak{d}(\mu, pw(x))<\epsilon\}$ has positive Lebesgue measure (in particular the physical measures are physical-like).   The set $\mathcal{PL}=\mathcal{PL}(Leb)$ of physical-like measures is the smallest compact subset of measures containing $pw(x)$ for Lebesgue almost every point $x\in M$.  In other terms if one considers the closed valued map 
\begin{eqnarray*}
pw: & X\rightarrow &\mathcal{KM}(M),\\
&x\mapsto& pw(x)
\end{eqnarray*}  and its essential range $\overline{\Ima}_{Leb}(pw)$ then we have (see Appendix \ref{deuse})
\[\mathcal{PL}=\bigcup_{K\in \overline{\Ima}(pw)}K.\]
Instead of the Lebesgue measure we may consider any other Borel measure $m$ and define similarly $\mathcal{PL}(m)$ 
as the smallest compact subset of measures containing $pw(x)$ for $m$-almost every point $x\in M$. We let $PL(m)$ be the set with full $m$-measure given by points $x\in M$ with $pw(x)\in \overline{\Ima}_m(pw)$ (in particular $pw(x)\subset \mathcal{PL}(m)$.  When $m$ is  absolutely continuous with respect to another 
Borel measure $m'$ then $\mathcal{PL}(m)$ is a subset of $\mathcal{PL}(m')$. For a subset $E$ of $M$ with $\Leb(E)>0$ we let $\Leb_E$ be the probability measure induced on $E$. In particular we have $\mathcal{PL}(\Leb_E)\subset \mathcal{PL}(\Leb_F)$ for subsets $E\subset F$.

\subsection{Lyapunov exponents}
In this section we consider a  $C^1$ diffeomorphism $f$ of $M$.  We recall some background on Lyapunov exponents (see \cite{BP} for further details), which estimate the exponential growth  in $n$ of the derivative $df^n$ on the tangent bundle $TM$ of $M$.  
\subsubsection{Lyapunov exponents of a point}
The (forward) Lyapunov exponent of $(x,v)\in TM\setminus \{0\}$ is defined as follows 
\[\chi(f,x,v):=\limsup_{n\rightarrow +\infty}\frac{1}{n}\log \|d_xf^n(v)\|.\]
The function $\chi(x,\cdot):=\chi(f,x,\cdot)$ admits only finitely many values $\chi_1(x)>...>\chi_{p(x)}(x)$ on 
$TM\setminus \{0\}$ and  generates a filtration $ 0\subsetneq V_{p(x)}(x) \subsetneq \cdots
\subsetneq V_{1}=T_xM$ with  $V_i(x)=\{ v\in TM, \ \chi(x,v)\leq \chi_i(x)\}$. The function $p$  as well the vector spaces $V_i(x)$, $i=1,...,p(x)$ depend Borel 
measurably on $x$. We let $\chi(x)$ be the maximal Lyapunov exponent at $x$ :
\[\chi(x):=\max_i\chi_i(x)=\max_{(x,v)\in TM\setminus \{0\}}\chi(x,v).\]

For any positive  integer $k$ less than or equal to the dimension $d$ of $M$ we may define similarly the  maximal Lyapunov exponent $\chi^k$ of the map $\Lambda^kdf$ induced by $df$ on the $k$-exterior bundle $\Lambda^kTM$. In particular we have $\chi^1=\chi$. Finally we put for all $x\in M$ :
\[\Sigma\chi^+(x)=\max\left(0,\chi^1(x),\cdots, \chi^d(x)\right).\]

A point is said \textbf{ (forward) Lyapunov regular } when $\chi^k(x)=\sum_{i=1}^k\chi_i(x)$ for all $1\leq k\leq d$. By Oseledets theorem \cite{Ose} the set of Lyapunov regular points has full $\mu$-measure for any invariant measure $\mu$. However we are here mostly interested in the typical dynamical behaviour with respect to the Lebesgue measure (which is a priori not  invariant). In particular it may happen that the set of Lyapunov regular points has not full Lebesgue measure ( see e.g. \cite{Ott} for the eight attractor). We will never assume  Lyapunov regularity in the present paper. \\

We let $\overline{\Sigma\chi^+}$ (resp. $\overline{\chi^k}$ for $k=1,\cdots, d$) be the essential supremum of $\Sigma\chi^+$  (resp. $\chi^k$) with respect to the Lebesgue measure, in particular 
$\overline{\Sigma\chi^+}=\max(0, \overline{\chi^1},\cdots, \overline{\chi^d})$. 
Based on Yomdin's theory and the volume growth estimates due to Newhouse, Koslowski \cite{Kos} showed an integral formula for the topological entropy of a $C^\infty$ smooth system. This equality may be written as follows:
\begin{equation*}h_{top}(f)=\lim_n\frac{1}{n}\log \int \max_k\|\Lambda^k d_xf^n \| \, d\Leb(x).
\end{equation*}

By Jensen's inequality we have for all integers $n$
\[\log \int  \max_k\|\Lambda^kd_xf^n\| \, d\Leb(x)\geq \int \log  \max_k\|\Lambda^kd_xf^n\|\, 
d\Leb(x).\]
According to  Borel-Cantelli Lemma, for all $\gamma>0$,  the set $\{x\in M, \  \max_k\|\Lambda^kd_xf^n\|\geq e^{n(\overline{\Sigma\chi^+}-\gamma)}\}$ has Lebesgue measure larger than $e^{-n\gamma}$ for infinitely many $n$. Therefore we conclude that 
\begin{equation}\label{ko} h_{top}(f)\geq  \limsup_n \frac{1}{n} \int \log\max_k\|\Lambda^kd_xf^n\|\, d\Leb(x)\geq \overline{\Sigma\chi^+}.
\end{equation}

\subsubsection{Lyapunov exponents of invariants measures}
For an invariant measure $\mu$ we let for $i=1,...,d$ 
\[\chi_i(\mu):=\int \chi_i(x)d\mu(x)\] and 
\[\chi_i^{+}(\mu):=\int \max\left(0,\chi_i(x)\right) d\mu(x).\]

For a sequence of real numbers $(a_n)_n$ we let $\lim^\searrow_na_n$ the limit in $n$ of the sequence $(a_n)_n$ when the sequence is converging to $\inf_na_n$.  
The maximal Lyapunov exponent $\chi(\mu)=\max_i\chi_i(\mu)$ and its positive part $\chi^+(\mu)=\max_i\chi_i^{+}(\mu)$  satisfy 
\[\chi(\mu)=\lim^\searrow_n\frac{1}{n}\int \log\|d_xf^n\| \, d\mu(x)\] and 
\[\chi^{+}(\mu)=\lim^\searrow_n\frac{1}{n}\int \log^+\|d_xf^n\| \, d\mu(x).\]
Similarly the sum $\Sigma\chi^+(\mu)=\sum_i\chi_i^+(\mu)$ of all the positive Lyapunov exponents of $\mu$ satisfies

\[\Sigma\chi^{+}(\mu)=\lim^\searrow_n\frac{1}{n}\int \log^+\max_k\|\Lambda^kd_xf^n\| \, d\mu(x).\]
When $\mu$ is ergodic, we get from the subadditive ergodic theorem for all $i$ and $\mu$-almost every $x$ $$\chi_i(x)=\chi_i(\mu),$$ $$\chi(\mu)=\chi(x)=\lim_n\frac{1}{n}\log\|d_xf^n\|,$$ $$\Sigma\chi^+(\mu)=\Sigma\chi^{+}(x)=\lim_n\frac{1}{n}\log^+\max_k\|\Lambda^kd_xf^n\|.$$  
The functions $\mu\mapsto \chi(\mu),\chi^+(\mu), \Sigma \chi^+(\mu)$ define upper semicontinuous fonctions on $\mathcal{M}(M,f)$ (see e.g.  \cite{Bon}). 

We recall that for a $C^1$ diffeomorphism Ruelle's inequality \cite{Rue} gives the following upperbound of the metric entropy $h(\mu)$ of an invariant measure $\mu$
\[h(\mu)\leq \Sigma \chi^+(\mu).\]

An ergodic measure is said\textbf{ hyperbolic} when any of its Lyapunov exponent is  nonzero.

\subsubsection{Strong Lyapunov exponents}
In this paragraph we introduce a new kind of pointwise Lyapunov exponents which is related with the Lyapunov exponents of the empirical measures. We define the \textbf{strong positive maximal Lyapunov exponent}  as follows. First we let for all $p\geq 1$ and  for all $x\in M$
\[\lambda_p(x):=\limsup_{n}\frac{1}{n}\sum_{l=0}^{n}\log ^+\|d_{f^{l}x}f^p\|.\]
Clearly we have $\frac{1}{p}\lambda_p(x)\geq \chi^+(x)$  by submultiplicativity of the norm. Moreover the sequence $(\lambda_p(x))_p$ is  a subadditive sequence.  Then we let for all $x\in M$
 \[\lambda(x):=\lim^\searrow_p \frac{1}{p}\lambda_p(x)\geq \chi^+(x).\]
 
\begin{lemma}Let $(M,f)$  be a  $C^1$ dynamical system. Then we have  for all $x\in M$
\[\sup_{\mu\in pw(x)}\chi^+(\mu)=\lambda(x).\]
\end{lemma}

\begin{proof}
Let $x\in M$. Let $\mu=\lim_k\mu_{n_k}^x\in pw(x)$ for an increasing sequence of integers $(n_k)_k$. 
For all $n$ and $p$ we have 
\[\int \log^+\|d_yf^p\|\, d\mu_n^x(y)=\frac{1}{n}\sum_{l=0}^{n}\log ^+\|d_{f^{l}x}f^p\|.\]
Taking the limit over $n=n_k$ when $k$ goes to infinity  we get 
\[\int \frac{\log^+\|d_yf^p\|}{p}\, d\mu(y)\leq \frac{\lambda_p(x)}{p}\]
and by taking the limit when $p$ goes to infinity we have finally
\[\sup_{\mu\in pw(x)}\chi^+(\mu)\leq \lambda(x).\]

Let us now show $\sup_{\mu\in pw(x)}\chi^+(\mu)\geq \lambda(x)$. For any $p$ there exist a subsequence $(n_{k,p})_k$ such that 
\[\lambda_p(x)=\lim_k\frac{1}{n_{k,p}}\sum_{l=0}^{n_{k,p}}\log ^+\|d_{f^{l}x}f^p\|.\]
Then if $\mu_p\in pw(x)$ is a weak limit of $(\mu_{n_{k,p}}^x)_k$ we have 
\[ \int \log^+\|d_yf^p\|\, d\mu_p(y)= \lambda_p(x) .\]

For any $\mu\in \mathcal{M}(M,f)$ and for any $z\in X$, the sequences $\left(\int \log^+\|d_yf^p\|\, d\mu(y)\right)_p$ and $(\lambda_p(z))_p$ being both subadditive the terms $\chi^+(\mu)$ and $\lambda(z)$ are respectively the limits of the nonincreasing sequences $\left(\frac{ \int \log^+\|d_yf^{p_k}\|\, d\mu(y)}{p_k}\right)_k$ and $\left(\frac{\lambda_{p_k}(z)}{p_k}\right)_k$ for any increasing sequence of integers $(p_k)_k$ with $p_k \mid  p_{k+1}$ for all $k$.

 Fix such a sequence $(p_k)_k$. We get  :
\begin{eqnarray*}\sup_{\mu\in pw(x)}\chi^+(\mu)&=&\sup_{\mu\in pw(x)}\inf_p\int \frac{\log^+\|d_yf^{p}\|}{p}\, d\mu(y),\\
&=& \sup_{\mu\in pw(x)}\inf_k\int \frac{\log^+\|d_yf^{p_k}\|}{p_k}\, d\mu(y),\\
& \stackrel{\text{Proposition 2.4 in \cite{BD}}}{=}&\inf_{k}\sup_{\mu\in pw(x)}\int \frac{\log^+\|d_yf^{p_k}\|}{p_k}\, d\mu(y),\\
&\geq &\inf_k \int \frac{\log^+\|d_yf^{p_k}\|}{p_k}\, d\mu_{p_k}(y),\\
&\geq &\inf_k \frac{\lambda_{p_k}(x)}{p_k},\\
\sup_{\mu\in pw(x)}\chi^+(\mu)&\geq & \inf_p \frac{\lambda_p(x)}{p}=\lambda(x).
\end{eqnarray*}
\end{proof}

Similarly we may define  the \textbf{strong positive  sum of  Lyapunov exponents} as \[\Sigma \lambda(x):=\lim^\searrow_p\limsup_{n}\frac{1}{np}\sum_{l=0}^{n}\log ^+\max_k\|\Lambda^kdf^p\|_{f^{l}x}\geq \Sigma\chi^+(x).\] 

Following the above proof we get in a similar way :
\begin{lemma}\label{df}
Let $(M,f)$  be a  $C^1$ dynamical system. Then we have  for all $x\in M$ 
\[\Sigma\lambda(x)=\sup_{\mu\in pw(x)}\Sigma\chi^+(\mu).\]
\end{lemma}

A  point $x$ is said to be \textbf{ regular} when  we have  $\Sigma\lambda(x)= \Sigma\chi^+(x)>0.$ 
For an ergodic measure $\mu$, almost every point $x$ with respect to $\mu$ lies in the basin $\mathcal{B}_\mu$ of $\mu$ (in other terms $pw(x)=\mu$) and  $\Sigma\chi^+(x)=\Sigma\chi^+(\mu)$. Using the ergodic decomposition it follows then from Lemma \ref{df} :

\begin{lemma}
Regular points have full measure with respect to any invariant measure. 
\end{lemma}

However as already said we are interested in empirical measures with Lebesgue typical initial conditions and we do not assume there exists an invariant measure absolutely continuous with respect to $\Leb$. We denote by $\overline{\Sigma\lambda}$  the essential supremum of $\Sigma\lambda$ with respect to $\Leb$. As the set $PL:=PL(\Leb)$ has full Lebesgue measure we have 
\[\overline{\Sigma\lambda}\leq \sup_{x\in PL}\Sigma\lambda(x)\]
and then it follows from Lemma \ref{df} and $\Sigma\lambda\geq \Sigma \chi^+$ that :
\[\overline{\Sigma \chi^+}\leq \sup_{\mu\in \mathcal{PL}}\Sigma\chi^+(\mu).\]
In general  the equality does not hold as it can be seen again on the eight attractor  \cite{Ott,Lsy},  where we have $0=\overline{\Sigma \chi^+}<\chi^+(\delta_S)=\sup_{\mu\in \mathcal{PL}}\Sigma\chi^+(\mu)$ with $S$ being the associated saddle hyperbolic point.

\subsection{Sinai-Ruelle-Bowen measures}
We consider here a $C^{1+\alpha}$ diffeomorphism $f$ of $M$. An invariant measure $\mu$ is said to be a \textbf{Sinai-Ruelle-Bowen measure} (SRB measure for short) when $\mu$-almost every point has a positive Lyapunov exponent and the disintegration of $\mu$ along the unstable manifolds is absolutely continuous with respect to the volume on the unstable manifolds inherited from the Riemanian structure on $M$.

 From Pesin theory any ergodic hyperbolic SRB measure is physical \cite{Pes}. For an invariant measure $\mu$ of a $C^{1+\alpha}$ diffeomorphism we let $T_\mu$ be the set of (forward) Lyapunov regular points $x$ in the basin $\mathcal{B}_\mu$ of $\mu$ with $\chi_i(x)=\chi_i(\mu)$ for all $i$. In particular any point $x$ in $T_\mu$ satisfies $\Sigma\chi^+(x)=\sum_i\max(0,\chi_i(x))=\Sigma\chi^+(\mu)$ and therefore any such point  is regular in our sense.   Tsujii showed that there exists an  SRB measure when the union of $T_\mu$ over all ergodic hyperbolic measures has positive Lebesgue measure.  He also proved that an ergodic hyperbolic measure $\mu$   is an SRB measure if and only if $T_\mu$ has positive Lebesgue measure.

 Ledrappier and Young \cite{LY} (see also \cite{Qua} for the noninvertible version) gave a thermodynamical characterization of SRB measures : they are exactly the invariant measures with a positive Lyapunov exponent  almost everywhere satisfying the so-called Pesin formula :
\[h(\mu)=\int \sum_i \chi_i^+(x)d\mu(x).\] In particular any SRB measure has positive entropy. It is thus hyperbolic when considering a surface diffeomorphism.  The set of SRB measures is  a face of the Choquet simplex of invariant measures, i.e. the ergodic components of a SRB measure are also SRB measures.  
 As a direct consequence of the aforementioned  results we have for any   $C^{1+\alpha}$ surface diffeomorphism :
 \begin{align*} \sup_{\mu \ \text{SRB}}h(\mu)&\leq \sup_{\mu\text{ physical}}h(\mu),\\
 &\leq  \sup_{\mu\text{ physical}}\chi(\mu),\\
 & \leq \overline{\chi}.
 \end{align*}
 
 \begin{ques}
Do we have  $\sup_{\mu \in \mathcal{PL}}h(\mu)\leq \overline{\chi}$ for a $C^1$ (resp. $C^{1+\alpha}$, $C^\infty$) surface diffeomorphism?
 \end{ques}



\section{Statements }

We aim to compare the entropy of physical-like measures with the (strong) positive  sum of  Lyapunov exponents for $C^\infty$ systems. 

\begin{theorem*}
Let $f:M\rightarrow M$ be a $C^{\infty}$ map. Then for Lebesgue almost every point $x$ there exists $\mu_x\in pw(x)$ with  \[(\Sigma\lambda(x)\geq )\  h(\mu_x)\geq \Sigma\chi^+(x).\]
\end{theorem*}

Of course the inequality does not hold true for all $x$, e.g. when $x$ is a periodic point with a positive Lyapunov exponent. However the set of such points has zero Lebesgue measure.

\begin{rem}
For a $C^2$ Axiom A diffeomorphism $f:M\rightarrow M$, there are finitely many ergodic physical measures whose basins cover a set of full Lebesgue measure. Such measures also satisfies Pesin  formula. In this case we have moreover $\Sigma\chi^+(x)=\int \log \Jac(df|_{E_u})(x)\, d\mu(x)$ for $x\in \mathcal{B}_\mu$ by continuity of $x\mapsto\Jac(df|_{E_u})(x)$. Therefore  for Lebesgue almost every point $x$ we get
$ h(\mu_x)= \Sigma\chi^+(x) \text{ with }pw(x)=\{\mu_x\}$.
\end{rem}

 As a direct consequence of the Main Theorem we obtain the following lower bound on the entropy of a physical measure. 

\begin{coro}\label{phy}
Let $\mu$ be a physical measure of a $C^\infty$ map $f:M\rightarrow M$. Then 
\[h(\mu)\geq \overline{\Sigma\chi^+|_{\mathcal{B}_\mu}},\]
where $\overline{\Sigma\chi^+|_{\mathcal{B}_\mu}}$ is the essential supremum of $\Sigma\chi^+$ on $\mathcal{B}_\mu$. 
\end{coro}

The Main Theorem and Corollary \ref{phy} are wrong in finite smoothness. We give in the Appendix A an example of a $C^r$ smooth interval map for any finite $r\geq 1$
with a Dirac physical  measure at a source such that the essential supremum of the Lyapunov exponent on its basin  is positive.\\

We recover Inequality (\ref{ko}) obtained from Kozlovski integral formula.  More precisely we have :

\begin{coro}
Let $f:M\rightarrow M$ be a $C^\infty$ map. Then
\[\max_{\mu\in \mathcal{PL}}h(\mu)\geq \overline{\Sigma\chi^+}.\]
\end{coro}

\begin{proof}
For any $\epsilon>0$ the set $\{\Sigma\chi^+>\overline{\Sigma\chi^+}-\epsilon\}$ has positive Lebesgue measure, so that there exists a point  $x$ in this set with $pw(x)\subset \mathcal{PL}$ which satisfies the conclusion of the Main Theorem, i.e. there exist $\mu_x\in pw(x)$ with 
\[h(\mu_x)\geq \Sigma\chi^+(x)>\overline{\Sigma\chi^+}-\epsilon.\]
We conclude by  upper semicontinuity of the metric entropy for $C^\infty$ maps \cite{new} and by compactness of $\mathcal{PL}$.
 \end{proof}


For $C^\infty$ maps 
 we get the following refinement of Tsujii's theorem. 

\begin{coro}\label{sd}
Let $f:M\rightarrow M$ be a $C^\infty$ map.
\begin{enumerate}
\item Assume the set of  regular points in $\{\Sigma\chi^+>0\}$ has positive Lebesgue measure. Then $f$ admits an SRB measure. 
\item Let $\mu$ be a physical measure such that the set of  regular points in $\{\Sigma\chi^+>0\}\cap \mathcal{B}_\mu$ has positive Lebesgue measure. Then $\mu$ is an SRB measure. 
\end{enumerate}
\end{coro}

We recall Tsujii's results only deal with diffeomorphisms but under the weaker $C^{1+\alpha}$ smoothness assumption.  Contrarily to Tsujii's statement  we do not assume in the second item neither ergodicity nor hyperbolicity of the physical measure $\mu$. 

\begin{proof} We only prove the first item. The proof of the second one follows the same lines.
According to the Main Theorem, for Lebesgue almost every $x$ in $\{\Sigma\chi^+=\Sigma\lambda>0\}$ there is an SRB measure $\mu_x\in pw(x)$    satisfying
\[h(\mu_x)\geq \Sigma\chi^+(x).\] 
Moreover it follows from Ruelle's inequality and Lemma \ref{df} that 
\[\Sigma\lambda(x)\geq \Sigma\chi^+(\mu_x)\geq h(\mu_x).\]
Since we have $\Sigma\chi^+(x)=\Sigma\lambda(x)$ the measure $\mu_x$ satisifes Pesin's entropy formula and is therefore an SRB measure.
\end{proof}

Unlike the Main Theorem, which is false in finite smoothness,  we conjecture Corollary \ref{sd} holds true for any $C^{1+\alpha}$ map. It can be deduced from the Reparametrization Lemma in \cite{Bur} the case of $C^{1+\alpha}$ interval maps and surface diffeomorphisms. However as it involves stronger technicalities we prefer to only consider $C^{\infty}$ maps in the present paper.  

Observe also that the $C^\infty$ asumption does not imply that the basin of an ergodic physical measure contains a positive Lebesgue measure subset of regular points. If we consider again the eight attractor of Bowen \cite{Lsy} the strong Lyapunov exponent $\Sigma\lambda(x)$ is equal to the unstable Lyapunov exponent of the saddle physical measure, whereas according to our Main Theorem  we have $\Sigma\chi^+(x)=0$ for Lebesgue almost every point in the basin.   

\section{Some technical lemmas}

For a $C^2$ Anosov surface diffeomorphism one build SRB measures as follows. One takes the inherited Lebesgue measure $\mu$ on  a local unstable manifold and then  checks that the limit $\nu$ of $\left(\frac{1}{n}\sum_{0\leq k<n}f^k\mu\right)_n$ disintegrates absolutely continuously on unstable manifolds with respect to the Lebesgue measure. Here we follow somehow a similar approach by considering the Lebesgue measure  $\mu$ on a smooth disc with Lebesgue typical exponential growth. Then we estimate the entropy of $\nu$ by using a  \textit{Reparametrization Lemma} of dynamical balls.

\subsection{Lyapunov exponent along smooth leaves}
In the Lemma below we select the appropriate smooth disc.
\begin{lemma}\label{fol} Let $1\leq k\leq d$ and let $a<\overline{\chi^k}$. We consider a Borel subset $E$ of $\{a<\chi^k \}$ with positive Lebesgue measure. Then there exist a compact  subset $F$ of $E$ and a  foliation box $U$ with respect to a $C^\infty$ smooth $k$-foliation $\mathcal{F}$  with $\Leb(U\cap F)>0$ such that  \[\forall x\in U\cap F, \     \chi^k(x,T_x\mathcal{F})> a,\] where $T_x\mathcal{F}$ denotes any unit-norm element of $\Lambda^k(TM)$ generating the tangent space at $x$ of the $\mathcal{F}$-leaf containing $x$. 
\end{lemma}
\begin{proof} We may assume $k=1$ without loss of generality. Let $F$ be a compact subset of $E$ with $\Leb(F)>0$ such that 
$x\mapsto V_i(x)$ is continuous on $F$ for all $i$, where $(V_i(x))_i$ denote the Lyapunov subspaces at $x$. Let $x$ be a Lebesgue density point of $F$ and let  $u\in V_1(x)\setminus\left(\bigcup_{i\geq 2} V_i(x)\right)$. We denote the exponential map at $x$ by $\exp_x:T_xM\rightarrow M$. Then for a small enough neighborhood $U$ of $x$ the vector $\left(d\exp_x(u)\right)_y$ belongs to  $V_1(y)\setminus\left(\bigcup_{i\geq 2} V_i(y)\right)$ for all $y\in U\cap F$. Finally this vector generates the tangent space at $y$ of the foliation $\mathcal{F}=\exp_x(\mathcal{F}_x)$ where $\mathcal{F}_x$ is the foliation in $u$-directed lines of $T_xM$. 
\end{proof}

\subsection{Entropy computation}
We state now a technical entropy computation due to Misiurewicz \cite{Mis} in its elementary proof of the variational principle for the entropy, which we will use to  bound  from below the entropy of $\nu$. 
For a probability space $(X,\mathcal{B},\mu)$ and a finite measurable partition $P$ of $X$ we denote the static entropy of $P$ as follows  \\\[H_\mu(P):=-\sum_{A\in P}\mu(A)\log \mu(A).\]

\begin{lemma}\cite{Mis}\label{fdf}Let $(X,f)$ be a Borel system. We consider a sequence $(\mu_n)_n$ of probability Borel measures on $X$ and the associated sequence $(\nu_n)_n$ given for all $n>0$ by $\nu_n=\frac{1}{n}\sum_{0\leq k<n}f^k\mu_n$. Then  we have with $P^n=\bigvee_{k=0}^{n-1}f^{-k}P$
\[\forall m>0, \ \frac{1}{m}H_{\nu_n}(P^m)\geq \frac{1}{n}\left(H_{\mu_n}(P^n)-3m\log\sharp P\right).\]
\end{lemma}

\subsection{Local distortion}
The key argument which allows to control the distortion is given by the following lemma whose proof relies on tools of semi-algebraic geometry.  
For $x\in M$, $n\in \mathbb{N}$ and $\alpha>0$ we let $B_f(x,n,\alpha)$ be the dynamical ball  at $x$ of length $n$ and size $\alpha$ :
\[B_f(x,n,\alpha):=\{y\in M, \, \mathsf d(f^lx,f^ly)<\alpha \text{ for }l=0,...,n-1\}.\]

\begin{Lemma}\label{six}Let $f:M\rightarrow M$ be a $C^\infty$ map. Let $a\in \mathbb{R}$, $\gamma\in \mathbb{R}^+\setminus \{0\}$ and let $k$ be a positive integer with  $k\leq d$. For some $\alpha>0$, for all $x\in M$ and for all $\sigma:[0,1]^k\rightarrow M$ of class $C^\infty$ with $\|d\sigma\|\leq 1$ and $\Lambda^kd_t\sigma\neq 0$ for all $t\in [0,1]^k$, there exists  for large enough $n$ (depending on $\sigma$ but not on $x$) a family of  reparametrizations $(\theta_i^n:[0,1]^k\circlearrowleft)_{i\in I_n}$ with the following properties:
\begin{itemize}
\item $\bigcup_{i\in I_n} \Ima(\theta_i^n)\Supset \footnote{By $\bigcup_{i\in I}A_i\Supset B $ we mean that $\bigcup_{i\in I}A_i\supset B $ and $A_i\cap B\neq\emptyset$ for all $i\in I$.} \left\{t\in [0,1]^k,  \ \frac{\|\Lambda^kd_t(f^n\circ\sigma)\|}{\|\Lambda^kd_t\sigma\|}\geq e^{na} \text{ and } \sigma(t)\in B(x,n,\alpha)\right\},$ 
\item $\forall i\in I_n,\ \|d(f^n\circ \sigma \circ \theta^n_i)\|\leq 1$,
\item $\forall i\in I_n \ \forall t,t'\in \Ima(\theta^n_i), \ \frac{\|\Lambda^kd_t(f^n\circ \sigma)\|}{\|\Lambda^kd_{t'}(f^n\circ\sigma)\|}\leq 2$,
\item $\sharp I_n\leq e^{\gamma n}$.
\end{itemize}
\end{Lemma}

Such Reparametrizations Lemmas first appear in the pioneering work of Yomdin \cite{Yom} (see also \cite{Gr}) in his proof of Shub's entropy conjecture for $C^\infty$ systems. In Yomdin's earlier form the control of the distortion given by the third item did not appear. Moreover the reparametrized set was the whole dynamical ball (here this is the case when $f$ is a local diffeomorphim by choosing $a$ small enough). Others similar forms of the Reparametrization Lemma were used succesfully by the author to study symbolic extensions and exponential growth of periodic points for $C^r$ surface diffeomorphisms \cite{Bur, Burr}. The technical proof could be skipped at a first reading.\\

We first establish a version of the Reparametrization Lemma for a $C^\infty$ nonautonomous system $\mathfrak{F}=(\mathfrak f_l:B\rightarrow \mathbb{R}^d)_{l\in \mathbb{N}}$   on the unit Euclidean  ball $B$ of $\mathbb{R}^d$. For $m\in \mathbb{N}$ we let $\mathfrak{F}_m$ be the finite sequence of $C^\infty$ maps $\mathfrak{F}_m:=(\mathfrak f_l)_{0\leq l<m}$. In this context we define the dynamical ball  $B_{\mathfrak{F}_m}$  as follows
$$B_{\mathfrak{F}_m}:=\{y\in B, \ \mathfrak f_{l}\circ \cdots \circ \mathfrak f_0(y)\in B \text{ for }0\leq	l <m\}.$$ 
We  then put  $\mathfrak f^{m+1}=\mathfrak f_{m}\circ \cdots \circ \mathfrak f_0:B_{\mathfrak{F}_m}\rightarrow \mathbb{R}^d$ (let also  $\mathfrak f^0$ be the identity map of $\mathbb{R}^d$). 

Let $\mathcal A= (a_l)_{l\in \mathbb N}$ be an infinite sequence of integers. For the corresponding finite sequences  $\mathcal{A}_m:=(a_0,...,a_{m-1})$  we also consider the following dynamical ball induced by $\mathfrak{F}_m$ on the $k$-exterior bundle of the tangent space  $T\mathbb{R}^d$ endowed with the norm induced by the Euclidean norm:
\begin{multline*}B^k(\mathcal{A}_m):=\{(y,v)\in \Lambda^k(T\mathbb{R}^d) , \  \|v\|=1, \  y\in B_{\mathfrak{F}_m}  \\ \text{ and }  \forall l=0,...,m-1,\,  \lceil\log\|\Lambda^kd_{\mathfrak f^ly}\mathfrak f_l(v_l)\|\rceil=a_l \},
\end{multline*}
with the notations $v_l=\frac{\Lambda^kd_y\mathfrak f^l(v)}{\|\Lambda^kd_y\mathfrak f^l(v)\|}$,  $l=0,...,m-1$, and $\lceil\cdot\rceil$ for the ceiling function.

For a $C^\infty $ smooth $k$-disc $\mathfrak s:[0,1]^k\rightarrow \mathbb{R}^d$ we aim to reparametrize the set $C_\mathfrak{s}(\mathcal{A}_n)$ defined as follows :
$$C_\mathfrak{s}(\mathcal{A}_m)=\left\{t\in [0,1]^k,  \ \left(\mathfrak s(t),\frac{\Lambda^kd_t\mathfrak s}{\|\Lambda^kd_t\mathfrak s\|}\right)\in B^k(\mathcal{A}_m) \right\}.$$

\begin{prop}\label{auto}
With the above notations there  exists, for any integer $r>2$,  a family of reparametrizations  $(\phi^m_i:[0,1]^k\circlearrowleft)_{i\in \mathcal{I}(\mathcal{A}_m)}$  with the following properties :
\begin{enumerate}
\item $\bigcup_{i\in \mathcal{I}(\mathcal{A}_m)} \Ima(\phi^m_i)\Supset C_\mathfrak{s}(\mathcal{A}_m)$,
\item $\forall i\in \mathcal{I}(\mathcal{A}_m) \ \forall s=0,...,r $, $$\|d^s\left(\mathfrak f^m\circ \mathfrak s \circ \phi_i^m\right)\|\leq 1,$$
\item $\forall i\in \mathcal{I}(\mathcal{A}_m) \ \forall s=1,...,r-1$,  $$\|d^s\left(t\mapsto\Lambda^kd_{\phi^m_i(t)}(\mathfrak  f^m\circ \mathfrak s)\right)\|\leq \frac{1}{2}\max_{u\in [0,1]^k}\|\Lambda^kd_{\phi^m_i(u)}(\mathfrak f^m\circ \mathfrak s)\|,$$
\item $\sharp \mathcal{I}(\mathcal{A}_m)\leq C(r,d)^m\prod_{l=0}^{m-1}\max\left(1,\|d_0\mathfrak f_l\|^{k/r},\left(\frac{\max(1,\|\Lambda^kd_0 \mathfrak f_l\|)}{e^{a_l}}\right)^{\frac{k}{r-1}}\right)$ with $C(r,d)$  being a universal function in $r$ and $d$.
\end{enumerate}
\end{prop}

\begin{proof}
We argue by induction on $m$. Assume the family $(\phi^m_i:[0,1]^k\circlearrowleft)_{i\in \mathcal{I}(\mathcal{A}_m)}$ is already built for $\mathcal{A}_m=(a_0,\cdots, a_{m-1})$. We proceed to the inductive step by building the required family of reparametrizations with respect to $\mathcal{A}_{m+1}=(a_0,\cdots, a_{m})$. 
Observe that $\|d^{s+1}\mathfrak f_l\|\leq \alpha^s\|d^{s+1}f\|$ for all $l\in  \mathbb{N}$ and   $1\leq s\leq r-1$.
From the formula for the derivatives of a composition and the induction hypothesis we get therefore for  small enough $\alpha$ (depending only on $\|d^{s+1}f\|$, $s=1,...,r-1$) and for any $\phi=\phi_i^m$ :
\begin{align*}\|d^{r-1}\left(t\mapsto \Lambda^kd_{\phi(t)}(\mathfrak f^{m+1}\circ \mathfrak s)\right)\|&= 
\|d^{r-1}\left(t\mapsto\Lambda^kd_{\mathfrak f^{m}\circ \mathfrak s\circ \phi(t)}\mathfrak f_{m+1}\circ \Lambda^k d_{\phi(t)}(\mathfrak f^{m}\circ \mathfrak s)\right)\|,\\
&\leq A(r,d)\max\left(1,\|\Lambda^kd_{0}\mathfrak f_{m+1}\|\right)\max_t\|\Lambda^kd_{\phi(t)}(\mathfrak f^m\circ \mathfrak s)\|
\end{align*} 
and 
\begin{align*}\|d^{r}\left(\mathfrak f^{m+1}\circ \mathfrak s \circ \phi\right)\|&= 
\|d^{r-1}\left(d_{\mathfrak f^{m}\circ \mathfrak s\circ \phi}\mathfrak f_{m+1}\circ  d_\phi(\mathfrak f^{m}\circ \mathfrak s)\right)\|,\\
& \leq A(r,d)\max\left(1,\|d_{0}\mathfrak f_{m+1}\|\right),
\end{align*} 
for some universal function $A$ in $r$ and $d$. \\

We  use now the following lemma which is a slightly different version of the Main Lemma in \cite{Gr}. 
\begin{lemma}\label{alg}
Let $G_0:[0,1]^e\rightarrow \mathbb{R}^{e'}$ and $G_1:[0,1]^e\rightarrow \mathbb{R}^{e''}$ be respectively $C^r$ and $C^s$ maps. We denote by $B_{e'}$ and $B_{e''}$ the unit Euclidean balls of $\mathbb{R}^{e'}$ and $\mathbb{R}^{e''}$. Then there exists a family $(\psi_j:[0,1]^e\circlearrowleft)_{j\in  \mathcal  J}$ such that :
\begin{itemize}
\item $\bigcup_{j\in  \mathcal J}\Ima(\psi_j)\Supset G_0^{-1}(B_{e'}) \cap G_1^{-1}(B_{e''})$,
\item $\forall j\in  \mathcal J \ \forall k=0,...,r,\ \|d^k\left(G_0\circ\psi_j\right)\|\leq 1$,
\item $\forall j\in  \mathcal J \ \forall k=0,...,s, \ \|d^k\left(G_1\circ\psi_j\right)\|\leq 1/12$,
\item $\sharp  \mathcal J\leq B(r,s,e,e',e'')\times \max\left(1,\|d^rG_0\|^{e/r}, \|d^sG_1\|^{e/s}\right)$  for some universal function $B$.
\end{itemize}
\end{lemma}
The proof follows the same lines, the unique difference being that one applies the Algebraic Lemma  in \cite{Gr} simultaneously  to the interpolating polynomials of $G_0$ and $G_1$ with respective (maybe distinct)  degrees  $r$ and $s$.\\

To conclude the inductive step we apply   Lemma \ref{alg} for every $i\in \mathcal{I}(\mathcal{A}_m)$ with the $C^{r-1}$ map  $G_1:s\mapsto \frac{\Lambda^kd_{\phi_i^m(s)}(\mathfrak f^{m+1}\circ \mathfrak s)}{e^{a_{m}}\max_{u\in [0,1]^k}\|\Lambda^kd_{\phi_i^m(u)}(\mathfrak f^m\circ \mathfrak s)\|}$ and the $C^r$ map $G_0=\mathfrak f^{m+1}\circ \mathfrak s \circ \phi_i^m$ (for any $t\in \Ima(\phi_i^m)\cap  C_\mathfrak{s}(\mathcal{A}_{m+1})$ we have $\|G_1(t)\|\leq 1$).  We let $\psi_j$, $j\in \mathcal{J}=\mathcal{J}(\phi_i^m)$, be the resulting reparametrizations. 
The maps $\phi_{i,j}^{m+1}=\phi_i^m\circ \psi_j$ 
over all $(i,j)\in \mathcal{I}(\mathcal{A}_{m+1}):=\{(i,j), \ i \in \mathcal{I}(\mathcal{A}_m) \text{ and } j\in \mathcal{J}(\phi_i^m)\text{  with }\Ima(\phi_i^m\circ \psi_j)\cap C_{\mathfrak{s}}(\mathcal{A}_{m+1})\neq \emptyset\}$ 
 give the required family of reparametrizations for the $(m+1)^{th}$ step. Let us just check the new reparametrizations $\phi_{i,j}^{m+1}$ satisfies (3) for any  $ s=1,...,r-1$ :
\begin{align*}
\|d^s\left(t\mapsto\Lambda^kd_{\phi_{i,j}^{m+1}(t)}(\mathfrak f^{n+1}\circ \mathfrak s)\right)\|& \leq e^{a_{m}}\max_{u\in [0,1]^k}\|\Lambda^kd_{\phi_i^{m}(u)}(\mathfrak f^m\circ \mathfrak s)\|\|d^s(G_1\circ \psi_j)\|,\\
&\leq \frac{e^{a_{m}-1}}{4}\max_{u\in [0,1]^k}\|\Lambda^kd_{\phi_i^m(u)}(\mathfrak f^m\circ \mathfrak s)\|,\\
&\leq \frac{e^{a_{m}-1}}{2}\min_{u\in [0,1]^k}\|\Lambda^kd_{\phi_i^m(u)}(\mathfrak f^m\circ \mathfrak s)\|.
\end{align*}
Since we have $\Ima(\phi_{i,j}^{m+1}\circ \psi_j)\cap C_\mathfrak{s}(\mathcal{A}_{m+1})\neq \emptyset$ there exists $v\in [0,1]^k$,
with $\frac{\|\Lambda^kd_{\phi_{i,j}^{m+1}(v)}(\mathfrak f^{m+1}\circ \mathfrak s)\|}{\|\Lambda^kd_{\phi_{i,j}^{m+1}(v)}(\mathfrak f^{m}\circ \mathfrak s)\|}\geq e^{a_{m}-1}$ and therefore 
\begin{align*}\|d^s\left(t\mapsto\Lambda^kd_{\phi_{i,j}^{m+1}(t)}(\mathfrak f^{m+1}\circ \mathfrak s)\right)\|&\leq \frac{e^{a_{m}-1}}{2}\|\Lambda^kd_{\phi_{i,j}^{m+1}(v)}(\mathfrak f^{m}\circ \mathfrak s)\|,\\
&\leq \frac{1}{2}\|\Lambda^kd_{\phi_{i,j}^{m+1}(v)}(\mathfrak f^{m+1}\circ \mathfrak s)\|,\\
&\leq \frac{1}{2}\max_{u\in [0,1]^k}\|\Lambda^kd_{\phi_{i,j}^{m+1}(u)}(\mathfrak f^{m+1}\circ \mathfrak s)\|.
\end{align*}
This concludes the proof of Proposition \ref{auto}.
\end{proof}

A sequence $\mathcal{A}_m=(a_0,\cdots, a_{m-1})$ is said \textbf{$A$-admissible} for $A\in \mathbb{R}$ when 
  $$ B^k(\mathcal{A}_m)\cap  \{(y,v)\in \Lambda^k(T\mathbb{R}^d) , \ \|v\|=1 \text{ and }  \| \Lambda^kd_y\mathfrak  f^m(v)\|\geq e^{mA}\}\neq \emptyset.$$
In particular we have then $\sum_{l=0}^{m-1}a_l\geq mA$.

Let  $F$ be the real function $\mathbb{R}^+\ni t\mapsto t \left[ t^{-1}\log(t^{-1})+(1-t^{-1})\log(1-t^{-1})\right]$, in particular  $F(t)\leq t\log 2 $ for all $t$ and $\lim_{t\rightarrow +\infty}\frac{F(t)}{t}=0$.   By a standard combinatorial argument    (see e.g. Lemma 8 in \cite{Bur})  we have :
\begin{lemma}\label{combi}
Let $A\in \mathbb{R}$. Assume $\left|\log^+\|\Lambda^{k}d_0\mathfrak f_l\|-\log^{+}\|\Lambda^{k}d_y\mathfrak f_l\|\right|<1$ 
for all $l\in \mathbb{N}$ and $y\in B$. Then the number $k_n$ of $A$-admissible sequences $\mathcal{A}_m$ is bounded from above as follows 
\begin{align*}
\frac{\log k_m}{m}\leq F(\lambda^k(\mathfrak{F}_m)+2-A),
\end{align*}
where  $\lambda^k(\mathfrak{F}_m):=\frac{1}{m}\sum_{l=0}^{m-1}\log^+\|\Lambda^kd_0\mathfrak f_l\|$
\end{lemma}

We are now in position to prove the Reparametrization Lemma. 

\begin{proof}[Proof of the Reparametrization Lemma]Without loss of generality we can assume $a<-1$. Fix  then $\gamma>0$ and $x\in X$ and take   a positive integer $p$ precised later on.  Let $\mathbb{N}^*\ni n=p(m-1)+q$ with $m,q\in\mathbb{N}^*$ and $0< q\leq p$. As in the previous works \cite{Yom, Gr, Bur} we may replace\footnote{ Of course we only reparametrize  in this a way the subset $\sigma(\alpha[0,1]^k)$. But one can reparametrize similarly  $\sigma(C_\alpha)$ for any subcube $C_\alpha$ of $[0,1]^k$ of size $\alpha$ and we only need $\lceil \alpha^{-1}\rceil ^{d}$ such subcubes to cover $[0,1]^k$.} $\sigma$ by $\mathfrak s=\alpha^{-1}\sigma(\alpha\cdot)$ for $\alpha>0$ and the local dynamic of $f$ around $x$ of time $n$ by the nonautonomous system $\mathfrak{F}_m=(\mathfrak f_l)_{0\leq l < m }$  defined on the unit Euclidean  ball $B$ of $\mathbb{R}^d$ by  $\mathfrak f_l=\alpha^{-1}f^p(f^{pl}x+\alpha\cdot )$ for  $0\leq l< m-1$ and $\mathfrak f_{m-1}=\alpha^{-1}f^q(f^{p(m-1)}x+\alpha\cdot )$. We assume here without loss of generality that  $M$ is the $d$-torus $\mathbb{R}^d/\mathbb{Z}^d$ and $\alpha$ is less than $1$ (in general, without an affine structure, one should conjugate $f$ with the exponential map at $f^lx$  to get a map $\mathfrak f_l$ on $B\subset\mathbb{R}^d$ and take $\alpha$ less than the radius of injectivity of $(M,\|\cdot\|)$. Moreover one has to replace the Euclidean norm by the Riemanian norms along the orbit of $x$, in the nonautonomous system). 

We may take $\alpha>0$ so small that $\left|\log^+\|\Lambda^{k}d_0\mathfrak f_l\|-\log^{+}\|\Lambda^{k}d_y\mathfrak f_l\|\right|<1$   for all $0\leq l<m$ and $y\in B$. 
We have $p/2\leq n/m(\leq p)$ once $m\geq 2$. Therefore in this case a $an/m$-admissible sequence $\mathcal{A}_m$ is $ap/2$-admissible. It follows then from  Lemma \ref{combi} that the number $k_m$ of \textit{$an/m$-admissible} sequences $\mathcal{A}_m$ satisfies 
\begin{align*}
\frac{\log k_m}{m}\leq F(\lambda^k(\mathfrak{F}_m)+2-ap/2),
\end{align*}

Moreover  we have $$B_f(x,n,\alpha)\subset B_{\mathfrak{F}_m}$$ and 
for all $t\in \alpha [0,1]^k$ 
$$\frac{\|\Lambda^kd_t(f^n\circ\sigma)\|}{\|\Lambda^kd_t\sigma\|}=\frac{\|\Lambda^kd_{\alpha^{-1} t}(\mathfrak f^m\circ\mathfrak s)\|}{\|\Lambda^kd_{\alpha^{-1}t}\mathfrak s\|}.$$
Therefore we get 
\begin{multline*}\bigcup\{\alpha C_{\mathfrak{s}}(\mathcal{A}_m), \ \mathcal{A}_m \ an/m\text{-admissible }\} \supset \\  \left\{t\in \alpha[0,1]^k,  \ \frac{\|\Lambda^kd_t(f^n\circ\sigma)\|}{\|\Lambda^kd_t\sigma\|}\geq e^{na} \text{ and } \sigma(t)\in B(x,n,\alpha)\right\}.
\end{multline*}


For $\gamma>0$ we take $r$ such that 
\[\max(1,\|df\|^{k/r})\times \left(\frac{\max\left(1,\|\Lambda^kd f\|\right)}{e^{a}}\right)^{\frac{k}{r-1}}<e^{\gamma/6}.\]
 We consider then  an integer $p$ so large that   
$$ p>\frac{6\left(2k+\log C(r,d)\right)}{\gamma}$$
 and 
$$  \sup_{x>p\gamma/3\log 2}\frac{F(x)}{x-2}<\frac{\gamma}{6k\max(\log \|df\|,|a|)}.$$
This last constraint allows to control the number $k_m$ of $an/m$-admissible sequences  $\mathcal{A}_m$  (observe $\lambda^k(\mathfrak{F}_m)\leq pk\log^+ \|df\|$):
\begin{eqnarray*}\frac{\log k_m}{m}& \leq & F\left(\lambda^k\left(\mathfrak{F}_m\right)+2-ap/2\right),\\
&\leq & \max\left(\left(\lambda^k(\mathfrak{F}_m)-ap/2\right)\sup_{x>p\gamma / 3\log 2}\frac{F(x)}{x-2}, \sup_{x\leq p\gamma / 3\log 2}F(x)\right),\\
& \leq & \max\left(\frac{ p\gamma\left( k\log^+ \|df\|+|a|/2\right)}{6k\max(\log^+ \|df\|,|a|)}, p\gamma/3\right),\\
\frac{\log k_m}{m} &\leq & p\gamma /3.\end{eqnarray*}

Moreover,  for any $an/m$-admissible sequence $\mathcal{A}_m=(a_0,\cdots,a_{m-1})$ we have  $\frac{\max(1,\|\Lambda^kd_0\mathfrak{f}_l\|)}{e^{a_l}}\geq 1/e^2$ for any $0\leq l<m$ and therefore

\begin{align*}
&C(r,d)^m\prod_{l=0}^{m-1}\max\left(1,\|d_{ 0}\mathfrak f_l\|^{k/r},\left(\frac{\max(1,\|\Lambda^kd_{ 0}\mathfrak f_l\|)}{e^{a_l}}\right)^{\frac{k}{r-1}}\right)\\
&\leq (e^{2k}C(r,d))^m \prod_{l=0}^{m-1}\max\left(1,\|d_{ 0}\mathfrak f_l\|^{k/r}\right)  \times \prod_{l=0}^{m-1}\left(\frac{\max(1,\|\Lambda^kd_{ 0}\mathfrak f_l\|)}{e^{a_l}}\right)^{\frac{k}{r-1}},\\ 
&\leq  (e^{2k}C(r,d))^{m}\left( \max(1,\|df\|^{k/r})\times \left(\frac{\max(1,\|\Lambda^kd f\|)}{e^{a}}\right)^{\frac{k}{r-1}}\right)^{n},\\
&\leq  e^{2k}C(r,d)e^{\gamma n/3},
\end{align*}
where the last inequality follows from  $m\leq 1 + \frac{n}{p}\leq 1+\frac{\gamma n}{6(2k+\log C(r,d))}$).

The reparametrizations $(\phi_i^m)_{i\in \mathcal{I}(\mathcal{A}_m)}$ built  with 
respect to $\mathfrak{F}_p$ over all $an/m$-admissible sequences $\mathcal{A}_m$ then satisfies the conclusion of the Reparametrization Lemma after a rescaling of size $\alpha$:

\begin{itemize}
\item
 \begin{flalign*}
 &\bigcup \left\{\alpha \phi_i^{m}([0,1]^k), \  i \in \mathcal{I}(\mathcal{A}_m) \text{ and } \mathcal{A}_m \ an/m\text{-admissible }\right\} \\
&\supset \bigcup\{\alpha C_{\mathfrak{s}}(\mathcal{A}_m), \ \mathcal{A}_m \ an/m\text{-admissible }\}\\
&\supset   \left\{t\in \alpha [0,1]^k,  \ \frac{\|\Lambda^kd_t(f^n\circ\sigma)\|}{\|\Lambda^kd_t\sigma\|}\geq e^{na} \text{ and } \sigma(t)\in B(x,n,\alpha)\right\}.
 \end{flalign*}
 By taking  a subfamily  we may assume the image of each reparametrization has a non empty intersection   with this last set.
\item $\forall \mathcal{A}_m, \  i \in \mathcal{I}(\mathcal{A}_m),$
\begin{eqnarray*}
\| d(f^n\circ \sigma\circ\alpha  \phi_i^{m}) \| & = &  \alpha \| d \left( \mathfrak f^m\circ \mathfrak s \circ  \phi_i^m\right)\|,\\
&\leq &\alpha<1.
\end{eqnarray*}
\item $\forall \mathcal{A}_m, \  i \in \mathcal{I}(\mathcal{A}_m),$ we get from Lemma \ref{auto}  (with the notation  $\| \Lambda^kd_{\phi}g\|:=\max_{u\in [0,1]^k}\| \Lambda^kd_{\phi(u)}g\|$ for maps $\phi:[0,1]^k\circlearrowleft$ and $g:[0,1]^k\rightarrow \mathbb{R}^d$ or $M$):
\begin{align*}
\|d\left(t\mapsto\Lambda^kd_{\alpha \phi_{i}^m(t)}(  f^n\circ \sigma)\right)\|  &= \|d\left(t\mapsto\Lambda^kd_{ \phi_{i}^m(t)}(  \mathfrak f^m\circ \mathfrak s)\right)\|,\\
&\leq \frac{1}{2}\|\Lambda^kd_{\phi^m_i}( \mathfrak f^m\circ \mathfrak s)\|,\\
&\leq \frac{1}{2}\|\Lambda^kd_{\alpha\phi^m_i}( f^n\circ \sigma)\|.
\end{align*}

Then it follows from  the mean value inequality : 
$$\forall t,t'\in [0,1]^k, \ \|\Lambda^kd_{\alpha\phi^m_i(t)}( f^n\circ \sigma)-\Lambda^kd_{\alpha\phi^m_i(t')}(  f^n\circ \sigma)\|\leq \frac{1}{2}\|\Lambda^kd_{\alpha\phi^m_i}( f^n\circ  \sigma)\|$$ and  by the triangular inequality 
$$\|\Lambda^kd_{\alpha\phi^m_i(t)}( f^m\circ \sigma)\| \geq \|\Lambda^kd_{\alpha\phi^m_i(t')}(  f^m\circ \sigma )\|-\frac{1}{2}\|\Lambda^kd_{\alpha\phi^m_i}( f^m\circ  \sigma)\|.$$ 
Finally we get by  taking the maximum over $t'\in [0,1]^k$ :
$$\|\Lambda^kd_{\alpha\phi^m_i(t)}( f^m\circ \sigma)\| \geq \frac{1}{2}\|\Lambda^kd_{\alpha\phi^m_i}( f^m\circ  \sigma)\|.$$
\item \begin{align*}
&\sharp   \{\phi^m_i, \  i \in \mathcal{I}(\mathcal{A}_m) \text{ and } \mathcal{A}_m \  an/m- \text{admissible}\} \\ 
&\leq  \sum_{\mathcal{A}_m \  an/m- \text{admissible}}\sharp \mathcal{I}(\mathcal{A}_m),\\
&\leq   e^{2k}C(r,d)re^{\gamma n/3} k_m,\\
&\leq e^{2k}C(r,d)re^{\gamma n/3}e^{\gamma mp/3},\\
& \leq e^{\gamma n} \text{ for $n$ large enough (with $p$ staying fixed).}
\end{align*} 
\end{itemize}

\end{proof}

\begin{rem}In the proof of the Main Theorem below, we  will only need to apply the Reparametrization Lemma for $a>0$. \end{rem}

\section{Proof of the Main Theorem}

For   $1\leq k\leq d$ and  $ a <\overline{\chi^k}$  we let $PL^k_a=PL(\Leb_{\{\chi^k>a\}})$ be the set of points $x$ in $M$ such that any measure in $pw(x)$ is physical-like with respect to the Lebesgue measure induced  on $\{\chi^k>a\}$.
\begin{prop}\label{maine}For any   $1\leq k\leq d$ and  $  a <\overline{\chi^k}$  we have 
\[\forall x\in  PL^k_a \ \exists \mu_x\in pw(x), \ \ h(\mu_x)\geq a.\]
\end{prop}

We first prove the Main Theorem assuming the above Proposition \ref{maine}. Let $A$ be a countable and dense subset of $\mathbb{R}^+$. The countable  intersection $E$ over $1\leq k\leq d$ and $a_k\in A$ of the sets $PL^k_{a_k}\cup \{\chi^k\leq a_k\}$ has full Lebesgue measure. 
Fix $x\in E$ and let us show that there exists $\mu_x\in pw(x)$ with $ h(\mu_x)\geq \Sigma\chi^+(x)$. We may assume $\Sigma\chi^+(x)>0$. Take $k$ with $\chi^k(x)=\Sigma\chi^+(x)$. For any $a_k\in A $ with $a_k<\chi^k(x)$ we have $h(\mu_x)\geq a_k$  for some $\mu_x\in pw(x)$,  according to Proposition \ref{maine}. Since $A$ is dense in $\mathbb{R}^+$ and the metric entropy is upper semicontinuous we conclude that  \[\sup_{\mu_x\in pw(x)}h(\mu_x)=\max_{\mu_x\in pw(x)}h(\mu_x)\geq 
\Sigma\chi^+(x).\]  

\begin{proof}[Proof of Proposition \ref{maine}]
Fix $x$ in $PL_a^k$. For all $\epsilon>0$ the set $E=\{y, \ \chi^k(y)>a \text{ and } \mathfrak{d}^H\left(pw(y),pw(x)\right)<\epsilon/2\}$ has positive Lebesgue measure. Let $F$ be the subset of $E$ and let $U$ be the $\mathcal{F}$-foliation box given both by Lemma \ref{fol}. Fix $\gamma,\epsilon>0$.  As the foliation is smooth, there is by Fubini's theorem a leaf  $L$ of $\mathcal{F}$ intersecting $F$ in a set of positive Lebesgue measure (for the Lebesgue measure $\Leb_L$ induced on the smooth leaf $L$). Let $\mathcal{V}$ be a finite cover of $pw(x)$ by balls $V$ of radius $\frac{\epsilon}{2}$ centered at $x_V\in pw(x).$ We put for all integers $n$ and for all $V\in \mathcal{V}$
\begin{eqnarray*}B_n^V(=B_n^V(x)):=&\{y\in L\cap F\subset U, \ \|\Lambda^k df^n(T_y\mathcal{F})\| \geq e^{na} &  \\
  & \text{ and } \mathfrak{d}(\mu_n^y,V)<\epsilon/2 \}.& \end{eqnarray*}
By Borel-Cantelli Lemma we have $Leb_L(B_n^{V'})\geq e^{-n\gamma}$ for some $V'\in \mathcal{V}$ and  for  $n$ in an infinite subset $I_{\epsilon,\gamma}$ of positive integers. Indeed ifnot we should have $\Leb_L(\limsup_nB_n^V)=0$ for all $V\in \mathcal{V}$, but as by Lemma 4 we have  $L\cap F\subset \{y, \ \chi^k(y,T_y\mathcal{F})>a \text{ and } \mathfrak{d}^H\left(pw(y),pw(x)\right)<\epsilon/2\}\subset \bigcup_{V\in \mathcal{V}}\limsup_n B_n^V$, 
it would contradict $\Leb_L(F)>0$. For $n\in I_{\epsilon,\gamma}$  we let  $\mu_n$ be the probability measure induced on $B_n^{V'}$ by the Lebesgue measure $\Leb_L$ on $L$ and $\nu_n:=\frac{1}{n}\sum_{l=0}^{n-1}f^l\mu_n=\int \mu_n^y\, d\mu_n(y)$.  By convexity of the metric $\mathfrak d$ we have  $\mathfrak d (\nu_n, pw(x))\leq \mathfrak d(\nu_n, x_{V'})<\epsilon$. 

\begin{lemma}\label{entb}
With the above notations, any weak limit $\nu=\nu_{\epsilon,\gamma}^{a,k}$ of $(\nu_n)_{n\in I_{\epsilon,\gamma}}$,  when $n\in I_{\epsilon,\gamma}$ goes to infinity, is $\epsilon$-close to $pw(x)$ and satisfies
\[h(\nu)\geq a-2\gamma.\]
\end{lemma}
We postpone the proof of Lemma \ref{entb}.
To conclude the proof of Proposition \ref{maine}   (admitting Lemma \ref{entb}) we consider a weak-limit $\mu$ of $\nu_{\epsilon,\gamma}^{a,k}$ when $\epsilon$ and $\gamma$ both go to zero. Clearly $\mu\in pw(x)$ and by upper semicontinuity of the metric entropy we get $h(\mu)\geq a$.
\end{proof}

\begin{proof}[Proof of Lemma \ref{entb}]
Let $\alpha$ be the scale given by the Reparametrization Lemma with respect to $\gamma$, $k$ and $a$. We consider a partition $P$ of $M$ with diameter less than $\alpha$. By standard arguments we may assume the boundary of $P$ has zero $\nu$-measure ; in particular the static entropy $\mu\mapsto H_{\mu}(P^m)$ is a continuous function for any $m$ at $\nu$.  By Lemma \ref{fdf}
 \[\forall m, \ \frac{1}{m}H_{\nu_n}(P^m)\geq \frac{1}{n}\left(H_{\mu_n}(P^n)-3m\log\sharp P\right).\]
By taking the limit when $n\in I_{\epsilon,\gamma}$ goes to infinity we get
\[\frac{1}{m}H_{\nu}(P^m)\geq \liminf_{n\in I_{\epsilon,\gamma}}\frac{1}{n}H_{\mu_n}(P^n).\]
Let $P^n_y$  being the element of the partition $P^n$ containing $y\in M$. Then we have 
 \[H_{\mu_n}(P^n)=\int -\log\mu_n(P^n_y)d\mu_n(y).\]
  We apply the Reparametrization Lemma at a given  point $y$  to a $C^\infty$ map $\sigma:[0,1]^k\rightarrow M$ parametrizing the leaf $L$. By taking the foliation box $U$ small enough we can assume $\|d\sigma\|\leq 1$ and $\Lambda^kd_t\sigma\neq 0$ for all $t\in [0,1]^k$. For $n$ large enough we let $\theta$ be the 
 resulting reparametrizations. The set $P^n_y\cap B_n^{V'}(x)\subset B(y,n,\alpha)\cap  B_n^{V'}(x)$ is covered by the images of the $\theta$'s. The Lebesgue measure of each $f^n\circ \sigma\circ
 \theta$ is bounded  from above by a universal constant $C$ according to the second item of the Reparametrization Lemma. From the first item and  the third item 
we get $\|\Lambda^k d_ {\theta(t)} (f^n\circ\sigma)\|\geq \|\Lambda^k d_{\theta(t)}\sigma\|e^{na}/2$ for any $t\in [0,1]^k$. Together  with the upperbound on the 
 number of reparametrizations given in  the last item we have for $n$ large enough (independently of $y\in M$) : \begin{align*}
 \Leb_L(P_y^n\cap B_n^{V'}(x))&\leq \sum_\theta \Leb((\sigma\circ \theta)([0,1^k])),\\
 &\leq \sum_\theta \int_{[0,1]^k} \|\Lambda^k d_{\theta(t)}\sigma\|\|\Lambda^k d_{t}\theta\|\, dt ,\\
 &\leq \sum_\theta  2e^{-na}\int_{[0,1]^k}  \|\Lambda^k d_{\theta(t)}(f^n\circ\sigma)\|\|\Lambda^k d_{t}\theta\|\, dt ,\\
 & \leq \sum_\theta  2e^{-na}\Leb((f^n\circ\sigma\circ \theta)([0,1^k])),\\
 &\leq 2Ce^{-na}\sharp\{\theta\},\\
  \Leb_L(P_y^n\cap B_n^{V'}(x))& \leq 2Ce^{-na}\times e^{\gamma n}.
\end{align*}
 But  for $n\in I_{\epsilon,\gamma}$  we have also $\Leb_L(B_n^{V'}(x))\geq e^{-n\gamma}$ so that we finally get for large enough $n\in I_{\epsilon,\gamma}$ and for all $y\in M$
 \[\mu_n(P^n_y)\leq 2Ce^{-na}\times e^{2\gamma n},\]
 \[H_{\mu_n}(P^n)\geq (a-2\gamma)n-\log (2C)\]
 and for all $m$ \[\frac{1}{m}H_\nu(P^m)\geq \liminf_{n\in I_{\epsilon,\gamma}}\frac{1}{n}H_{\mu_n}(P^n)\geq a-2\gamma.\]
By taking the limit in $m$ we conclude \[h(\nu)\geq a-2\gamma.\]
\end{proof}

\vspace{1,5cm}

\appendix
\section{Counter-example for $C^r$ interval maps for any finite $r$ }

For any positive integer $r$  we give an example of a $C^r$ (but not $C^{r+1}$) interval map $h:[0,3/2]\circlearrowleft$ such that 
for $x$ in a positive Lebesgue measure set the following properties hold: 
\begin{enumerate}
\item the empirical measures $(\mu_n^x)_n$ are converging  to the Dirac measure at a fixed point (therefore with zero entropy),
\item the Lyapunov exponent at $x$ satisfies $\chi(x)=\frac{\log \|h'\|_\infty}{r}>0$. 
\end{enumerate}
Consequently the Main Theorem does not hold true in finite smoothness. \\

\underline{\textit{Step 1:}} 
Let $\lambda>1$. We first consider a $C^r$ (even $C^\infty$) interval map $f:[0,3/2]\circlearrowleft$ with the following properties 
\begin{itemize}\item $f(0)=f(1)=0$,
\item $f$ has a tangency of order $r$ at $1$, i.e. $f^{(k)}(1)=0$ for $k=1,...,r$,
\item $f$ is affine with a slope equal to $\lambda=\|f'\|_{\infty}$ on the interval $[0, 1/\lambda]$.\\
\end{itemize}

\underline{\textit{Step 2:}}
After a small $C^\infty$ perturbation  of $f$ around $1$ we may build  a new map $g$ such that 
for some $n_0$ and $n\geq n_0$,  $g^k(1-1/n)$ lies in  $[0, 1/\lambda]$ for $k=1,...,r^n-1$ and  
$g^{r^{n}}(1-1/n)=1-1/n+1$. Indeed these conditions require  $g(1-1/n)=(1-1/n+1)\lambda^{-r^n+1}=o(1/n^r)$, so that one can choose $g$  arbitrarily $C^\infty$ closed to $f$ by taking $n_0$ large enough.  For the interval map $g$, the empirical measures at $1-1/n$ are converging to the Dirac measure at the fixed point $0$.  We may also assume $g$ is constant on $J_n:=[1-1/n, 1-1/n-1/2n^2]$ for $n\geq n_0$.
\begin{figure}[!ht]
\hspace{5cm}
\includegraphics[scale=0.2]{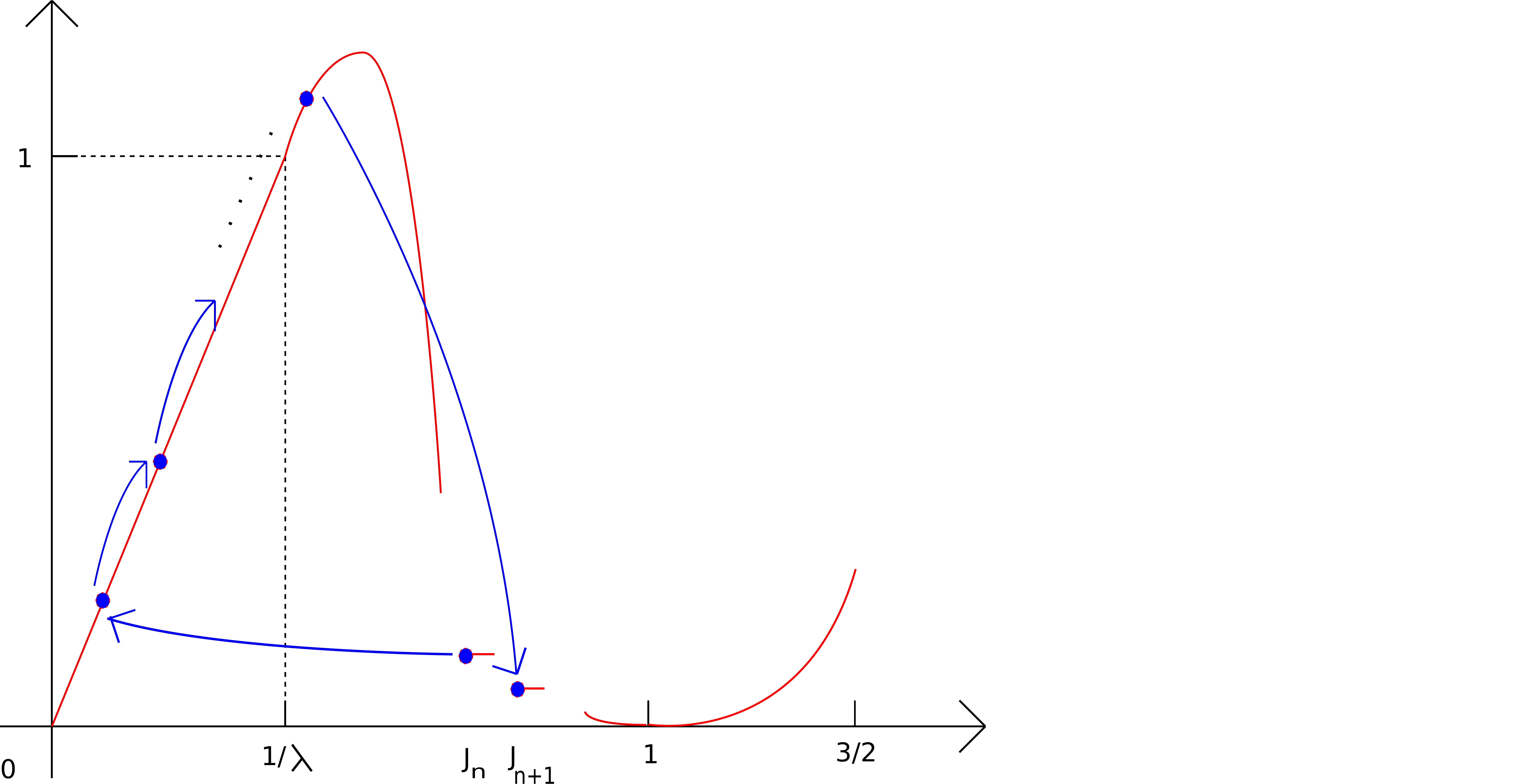}
\centering
\caption{\label{bl}\textbf{The graph of $g$ in red.} The arrows and points in blue  represent  the  orbit of $1-1/n\in J_n$. }
\end{figure}  

\underline{\textit{Step 3:}} We lastly modify  $g$ on $J_n$, $n\geq n_0$ such that the resulting map $h$ satisfies the desired properties. Let us first introduce an auxiliary family of functions $(f_p)_{p\in \mathbb{N}}$. For any $p$  we define  $f_p$ as the tent map $x\mapsto \max(x,1-x)$ on $[1/p,1/2-1/p]\cup [1/2+1/p,1-1/p]$. We extend it into  a $C^r$ smooth interval map in such a way $f_p$ vanishes and admits a tangency of order $r$ at the points $0,1/2$ and $1$. Finally we extend $f_p$ periodically on the whole real axis.  The intervals $[1/p,1/2-1/p]+k$ and $ [1/2+1/p,1-1/p]+k$ for $k\in \mathbb{Z}$ are called the affine branches of $f_p$.   Observe that the $C^r$ norm\footnote{The $C^r$ norm of a $C^r$ smooth interval map $f$ is the maximum over $k=0,...,r$ of the supremum norms $\|f^{(k)}\|_\infty$.} of $f_p$ may be chosen of order $p^{r}$. Then we let $h$ be $x\mapsto \alpha_nf_{n^2}\left((x-1+1/n) 2n^2N_n\right)+g(1-1/n)$ on $J_n$ where $\alpha_n\in \mathbb{R}^+$ and $N_n\in \mathbb{N}$ are chosen such that  
\begin{itemize}
\item for each affine branch $I_n$ in $J_n$,
$$h^{k}(I_n)\subset [0,1/\lambda]\text{ for }k=1,...,r^n-1$$ and $$h^{r^{n}}(I_n)=J_{n+1},$$
\item the $C^r$ norm of $h$ on $J_n$ goes to zero with $n$.
\end{itemize}
The first and second  conditions are respectively  fulfilled whenever 
 \[\lambda^{r^{n}-1}\times \alpha_n(1/2-2/n^2)= 1/2(n+1)^2\]and 
 \[ \max_{k=1,...,r} \|f^{(k)}_{n^2}\|_\infty\times \alpha_n\times (2n^2N_n)^r\sim n^{2r}\times \alpha_n\times (2n^2N_n)^r=1/n.\]
 
\begin{figure}[!ht]
\hspace{5cm}
\includegraphics[scale=0.16]{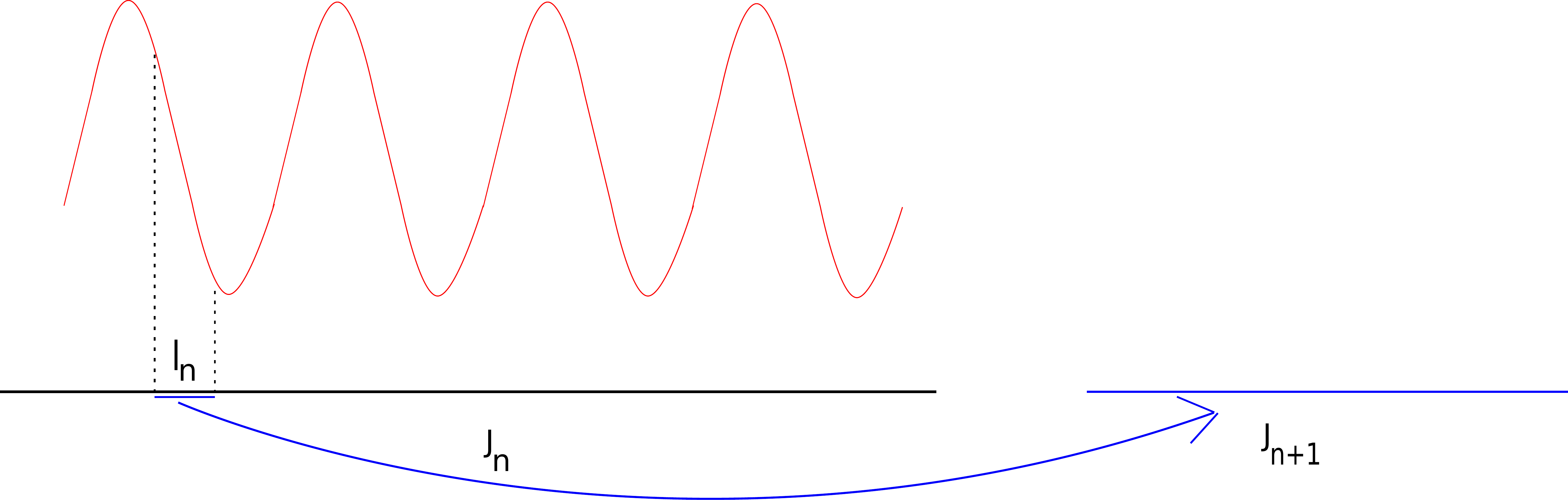}
\centering
\caption{\label{bll}\textbf{The graph of $h$ on $J_n$ in red.} The arrows and intervals in blue  represent   the image $J_{n+1}$ of an affine branch $I_n$ under $h^{r^n}$. }
\end{figure}

\underline{\textit{Conclusion:}} Let $E_n=\bigcup_{I_n}I_n$ be the union of affine branches in $J_n$ and let $E=E_{n_0}\cap  h^{-r^{n_0}}E_{n_0+1} \cap h^{-r^{n_0}-r^{n_0+1}}E_{n_0+2}\cap ...$ be the subset of points in $J_{n_0}$ visiting  successively the sets $E_n$, $n\geq n_0$. Clearly $E$ is contained in the basin of the Dirac measure at $0$. To conclude it remains to see  
that $E$ has positive Lebesgue measure and that $\chi(x)\geq \frac{\log \lambda}{r}$ for any $x$ in $E$. The set $E$ is an affine dynamically defined Cantor set where we remove a proportion  of $4/n^2$ at the $n^{th}$ step. Therefore $\Leb(E)=\Leb(E_{n_0})\prod_{n>n_0}(1-4/n^2)>0$. Finally as $\log |h'|$ is equal  on $I_n$ to $ \log(\alpha_n4n^2N_n)\sim \frac{r-1}{r}\log \alpha_n\sim -r^{n-1}(r-1)\log \lambda$,   the Lyapunov exponent at any $x\in E$ is given by
\begin{eqnarray*}
\chi(x)&=& \limsup_p\frac{1}{p}\log |(h^p)'(x)|,\\
&=&  \log \lambda \lim_q \frac{\sum_{q\geq n\geq n_0} \left(r^{n}-r^{n-1}(r-1)\right)}{\sum_{n\geq n_0} r^{n}},\\
&=& \frac{\log \lambda}{r}.
\end{eqnarray*}

\section{Essential range of $x\mapsto pw(x)$}\label{deuse}
We recall here the definition of the essential range of a Borel map with respect to a Borel measure. Finally we relate the set of physical-like measures of a topological system $(M,f)$ with the essential range of $M\ni x\mapsto pw(x)$.\\

We consider two metric spaces $X$ and $Y$ with $Y$ separable. Let $m$ be a Borel measure on $X$ and $\phi:X\rightarrow Y$ be a Borel map.

\begin{definition}
With the above notations the essential range $\overline{\Ima}_m(\phi)$ of $\phi$ with respect to $m$ is the complement 
of $\{y\in Y, \ \exists U \text{ open with }y\in U \text{ and }m(\phi^{-1}U)=0\}$.
\end{definition}

 The set $\overline{\Ima}_m(\phi)$ is a closed subset of $Y$ and for $m$-almost every $x$ the point $\phi(x)$ belongs to $\overline{\Ima}_m(\phi)$. Moreover it is the smallest set satisfying these  properties.
 
 \begin{lemma}Let $(M,f)$ be a topological system. The map  $pw: x\mapsto pw(x)$ from $M$ to $\mathcal{KM}(M)$ is Borel. 
 \end{lemma}
\begin{proof}
As the set $\mathcal{KM}(M)$ is  separable, it is enough to show $pw^{-1}(B)$ is a Borel subset of $M$ for any closed ball $B$ of $\mathcal{KM}(M)$. Let $B$ be the closed ball of radius $\epsilon$ centered at $K\in \mathcal{KM}(M)$, i.e. the set of compact subsets $K'$  of $M$ with $K'\subset K_{\epsilon}$ and  $K\subset K'_\epsilon$ where $K_\epsilon$ and $K'_\epsilon$ denote respectively the  closed  $\epsilon$-neighborhoods of $K$ and $K'$. Firstly observe that 
$\{x\in M, \ pw(x)\subset K_{\epsilon}\}$ is closed. Then for   a  fixed sequence $(k_n)_{n\in \mathbb{N}}$ dense in $K$ the following properties are equivalent :
\begin{eqnarray*}
 &   K\subset  (pw(x))_\epsilon,& \\
 \Leftrightarrow & \mathfrak d(k_n, pw(x))\leq \epsilon & \text{ for all }n,\\
 \Leftrightarrow & \liminf_p \mathfrak d (k_n,\mu_x^p)<\epsilon' & \text{ for all }n \text{ and }\mathbb{Q}\ni \epsilon'>\epsilon.
 \end{eqnarray*}
 The fonctions $x\mapsto  \mathfrak d (k_n,\mu_x^p)$ being continous we conclude that  $pw^{-1}(B)$ is a Borel set.
\end{proof}

\begin{lemma}
The set $\mathcal{PL}(m)$ of physical-like measure is the union of all $K\in \overline{\Ima}_m(pw)$.
\end{lemma}
\begin{proof}
Firstly, the set  $\overline{\Ima}_m(pw)$ being  a compact subset of $\mathcal{KM}(M)$, the set   $\bigcup_{K\in  \overline{\Ima}_m(pw)}K$ is a compact subset of $M$. Therefore, from the definitions we get $\mathcal{PL}(m)\subset \bigcup_{K\in  \overline{\Ima}_m(pw)}K$. We argue by contradiction to prove the converse inclusion. Assume there is $K\in \overline{\Ima}_m(pw)$ such that $K$ is not  contained in $\mathcal{PL}(m)$. Then this also holds for any $K'$ close enough to $K$. Therefore there exists an open neighborhood $U$ of $K$ such that $pw^{-1}(U)$ has positive $m$-measure and for all $x$ in this set $pw(x)$ is not contained in $\mathcal{PL}(m)$. It is impossible by definition of  $\mathcal{PL}(m)$.
\end{proof}

\vspace{1,5cm}

\bibliographystyle{amsplain}

\end{large}
\end{document}